\documentclass[a4paper,12pt]{article}
\usepackage[left=2.8 cm,top=2.2cm,bottom=3.2cm,right=2.8cm]{geometry}
\usepackage{amsfonts,amssymb,amsthm,amsmath}
\usepackage{mathrsfs}
\usepackage{bbm}
\usepackage{fontspec}

\usepackage{xcolor}

\definecolor{green(munsell)}{rgb}{0.0, 0.66, 0.47}
\definecolor{BlueGreenn}{rgb}{0.3,0.5,0.8}
\definecolor{DB}{rgb}{0.3,0.3,0.3}
\definecolor{DOr}{rgb}{0.7,0.3,0.3}
\definecolor{DGr}{rgb}{0.3,0.7,0.3}
\definecolor{DBl}{rgb}{0.1,0.3,0.5}
\definecolor{arylideyellow}{rgb}{0.91, 0.84, 0.42}
\definecolor{burntorange}{rgb}{0.8, 0.33, 0.0}
\definecolor{chromeyellow}{rgb}{1.0, 0.65, 0.0}
\definecolor{niceyellow}{RGB}{255,204,0}
\definecolor{lightgray}{rgb}{0.96,0.96,0.96}
\definecolor{darkgreen}{rgb}{0.1,0.45,0.06}
\definecolor{darkred}{rgb}{0.65,0.1,0.1}
\definecolor{lightred}{rgb}{1,0.9,0.9}
\definecolor{darkblue}{rgb}{0.1,0.1,0.5}
\definecolor{lightblue}{rgb}{0.88,0.92,0.99}
\definecolor{lightorange}{rgb}{1,0.95,0.89}
\definecolor{lightyellow}{rgb}{1,1,0.87}
\definecolor{burntlightorange}{RGB}{245,235,211}
\definecolor{skde}{RGB}{0,80,158}
\definecolor{skdelight}{RGB}{198,215,231}
\definecolor{skdeverylight}{RGB}{228,245,251}
\definecolor{skdelys}{RGB}{122,178,220}
\definecolor{pink}{RGB}{232,88,208}

\usepackage[hypertexnames=true, citecolor = burntorange, colorlinks = true, linkcolor = BlueGreenn, pdffitwindow = false, urlbordercolor = white,pdfborder={0 0 0}]{hyperref}%

\usepackage[T1]{fontenc}
\usepackage{microtype}

\usepackage{color}
\usepackage[utf8]{inputenc}

\usepackage{framed}
\usepackage{centernot}

\usepackage{tikz-cd}
\usepackage{tkz-euclide}

\usepackage[shortlabels]{enumitem}
\usepackage{dirtytalk}
\usepackage{subcaption}
\usepackage{float}
\usepackage{physics}
\usepackage{tcolorbox}
\usepackage{tikz}
\usetikzlibrary{
  math,
  bbox,
  knots,
  hobby,
  decorations.pathreplacing,
  shapes.geometric,
  calc
}

\usepackage{comment}

\usepackage{pgfplots}
\pgfplotsset{compat=1.18} 

\usepackage{mathtools}
\mathtoolsset{showonlyrefs=true} 

\usepackage{ytableau}
\ytableausetup{boxsize=0.5em,centertableaux}

\usepackage[normalem]{ulem}

\setlist{nolistsep}
\usepackage{titlesec}
\titleformat{\section}{\large\bfseries\color{DBl}}{\thesection}{1em}{}
\titleformat{\subsection}[runin]
{\normalfont\normalsize\bfseries\color{DBl}}{\thesubsection}{0.5em}{}[.]

\titlespacing\subsection{0pt}{12pt plus 4pt minus 2pt}{4pt plus 2pt minus 2pt}

\makeatletter
\patchcmd{\@maketitle}
  {\@title}
  {\color{\@titlecolor}\@title}
  {}{}
\newcommand{\@titlecolor}{DBl}
\newcommand{\titlecolor}[1]{\renewcommand{\@titlecolor}{#1}}
\makeatother

\numberwithin{equation}{section}

\newtheorem{theorem}{Theorem}[section]
\newtheorem{proposition}[theorem]{Proposition}
\newtheorem{lemma}[theorem]{Lemma}

\newtheorem{corollary}[theorem]{Corollary}

\theoremstyle{definition}
\newtheorem{definition}[theorem]{Definition}

\newcommand{\TL}{\text{TL}}

\newcommand{\ot}{\otimes}
\newcommand{\tp}[1]{^{\otimes #1}}    

\newcommand{\col}{\,:\,}

\usepackage{amsmath}

\DeclareMathOperator{\End}{End}

\DeclareMathOperator{\id}{id}

\DeclareMathOperator{\re}{Re}

\DeclareMathOperator{\diag}{diag}

\newcommand{\Cl}{\mathbb{C}}

\newcommand{\Nl}{\mathbb{N}}

\newcommand{\Ql}{{\mathbb Q}}
\newcommand{\Tl}{{\mathbb T}}

\newcommand{\R}{\mathcal{R}}

\newcommand{\CI}[0]{\mathcal{I}} 
 \newcommand{\CL}[0]{\mathcal{L}}
 \newcommand{\CN}[0]{\mathcal{N}}
 
 \newcommand{\CR}[0]{\mathcal{R}}

\newcommand{\la}{\lambda}

\DeclareMathAlphabet\EuScript{U}{eus}{m}{n}

\usepackage[style=alphabetic, labelnumber, defernumbers=true, firstinits=true, maxbibnames=99, url=false, doi=true, isbn=false, eprint=false, backend=biber]{biblatex}
\renewbibmacro{in:}{}
\DeclareFieldFormat{pages}{#1}
\ytableausetup{boxsize=0.5em,centertableaux}

\newcommand{\bT}{\mathbb{T}}
\usepackage{scalerel, tikz}

\renewcommand{\boxtimes}{\mathbin{\scalerel*{\tikz{\draw[line width=1.1pt](0,0)rectangle(1,1)--(0,0)(1,0)--(0,1);}}{\otimes}}}

\title{\bfseries The classification problem for unitary R-Matrices with two eigenvalues}

\author{Gandalf Lechner\footnote{Dep. Mathematik, FAU Erlangen-N\"urnberg, Cauerstr. 11, 91058 Erlangen. {\ttfamily gandalf.lechner@fau.de}}}
\date{\small\today}

\begin{document}
	
\maketitle

\begin{abstract}
    The problem of classifying all unitary R-matrices of arbitrary finite dimension that have precisely two distinct eigenvalues is described, working up to a natural equivalence relation given by the characters of their braid group representations. Up to one class that might or might not exist in even dimension larger than two, a full classification theorem is obtained.
\end{abstract}

\section{Introduction}

The (quantum) Yang-Baxter Equation (YBE) is an algebraic nonlinear equation for a linear map $R:V\ot V\to V\ot V$ on the tensor square of a vector space $V$, namely
\begin{align}\label{eq:YBE}
    (R\ot1)(1\ot R)(R\ot1)=(1\ot R)(R\ot 1)(1\ot R)
\end{align}
as an equation in $\End(V\ot V\ot V)$, where $1$ denotes $\id_V$. This equation plays a role in remarkably many different fields in mathematics and mathematical physics, such as braid group representations \cite{Jones:1987}, knot theory \cite{Turaev:1988}, subfactors \cite{Wenzl:1990}, quantum groups \cite{ChariPressley:1994}, Hopf algebras \cite{AndruskiewitschGrana:2003}, statistical mechanics \cite{Baxter:1982}, quantum field theory \cite{AbdallaAbdallaRothe:2001}, and quantum computing \cite{RowellWang:2017}. It also has several interesting variants, such as the classical Yang-Baxter equation \cite{BurbanKreussler:2012} or the set-theoretic Yang-Baxter equation \cite{CedoVendramin:2026}.

This widespread appearance of the YBE provides of course ample interest in understanding the solutions of \eqref{eq:YBE}, and has triggered approaches using various different tools ranging from algebraic and categorical methods over combinatorics to operator-algebraic ideas. Naive direct attempts to solve it run into the problem that \eqref{eq:YBE} amounts to a system of $(\dim V)^6$ cubic equations in $(\dim V)^4$ variables, which is not manageable even for low $\dim V>2$. From a more general point of view, the structure of the solution set of the YBE -- either for fixed $V$ or for a class of vector spaces -- is largely unknown.

One therefore typically restricts attention to subfamilies of solutions, such as solutions of dimension two \cite{Hietarinta:1992_8}, solutions representing the Temperley-Lieb algebra \cite{Bytsko:2015}, charge-conserving solutions \cite{HietarintaMartinRowell:2024}, set-theoretic solutions \cite{CedoVendramin:2026}, or unitary and involutive solutions \cite{LechnerPennigWood:2019}.

Motivated by applications in quantum physics \cite{RowellWang:2012,AlazzawiLechner:2016,RowellWang:2017,CorreaDaSilvaLechner:2023}, we focus in this article on \emph{unitary} solutions: We require that $V\cong\Cl^d$ is a finite-dimensional Hilbert space and $R:V\ot V\to V\ot V$ is unitary w.r.t. the canonical scalar product on $V\ot V$, i.e. we may view $R\in\End(V\ot V)\cong M_{d^2}$ as a unitary $(d^2\times d^2)$-matrix, where we use the standard identification $M_d\ot M_d\cong M_{d^2}$, and $M_d$ denotes the ${}^*$-algebra of complex $(d\times d)$-matrices. Such solutions of \eqref{eq:YBE} will be called \emph{unitary R-matrices}, and we shall write $\CR(V)$ for the set of all unitary matrices with base space $V$, and $\CR=\bigcup_V\CR(V)$ for the union over all finite-dimensional Hilbert spaces~$V$. We will refer to $d:=\dim V$ as the \emph{dimension of $R$} (although $R$ is an operator on $\Cl^{d^2}$), and write $\dim R:=\dim V$.

A classification of $\CR(V)$ for general $V$ is out of sight. We will therefore identify R-matrices which are identical as far as their braid group representations are concerned. To define this equivalence relation, we use the standard notation
\begin{align}\label{eq:Rk}
    R_k := 1\tp{(k-1)}\ot R\ot1\tp{(n-k-1)} \in\End(V\tp{n})
    ,\qquad 1\leq k\leq n-1,
\end{align}
(the total number $n$ of tensor factors will be clear from context). In case $R$ is invertible, we obtain representations of the braid groups $B_n$, $n\in\Nl$: Recalling that $B_n$ is presented on $n-1$ generators $b_1,\ldots,b_{n-1}$ by the relations
\begin{align}\label{eq:BnRelations}
    \begin{aligned}
        b_kb_l&=b_lb_k,\qquad &|k-l|&>1
        ,\\
        b_kb_lb_k&=b_lb_kb_l,\qquad &|k-l|&=1,
    \end{aligned}
\end{align}
one easily checks that the maps
\begin{align}
    \rho_R^{(n)} : B_n \to \End(V\tp{n}),\qquad b_k\mapsto R_k,
\end{align}
uniquely extend to $B_n$-representations on $V\tp{n}$. In case $R$ is unitary (i.e. $R\in\CR(V)$), this representation is of course unitary.

\begin{definition}
    Two unitary R-matrices $R,S\in\CR$ are \emph{equivalent}, denoted $R\sim S$, if for every $n\in\Nl$, the representations $\rho_R^{(n)}\cong\rho_S^{(n)}$ are unitarily equivalent.
\end{definition}

This definition focusses on the braid group representation aspects of R-matrices. Dropping the unitarity assumption, it also makes sense for invertible R-matrices (also on infinite-dimensional spaces), and we refer to \cite{AndruskiewitschGrana:2003,AlazzawiLechner:2016,LechnerPennigWood:2019} for some appearances of it in the literature.

Note that $\dim R$ is an invariant (i.e. $R\sim S\Rightarrow\dim R=\dim S$), and the spectrum $\sigma(R)$ of $R$ (including multiplicities) is also an invariant. This follows by considering the unitary equivalences $\rho_R^{(1)}\cong\rho_S^{(1)}$ and $\rho_R^{(2)}\cong\rho_S^{(2)}$ for two equivalent R-matrices $R\sim S$. The \emph{partial trace} of $R$ is another important invariant that will be recalled in the next section.

We also note that for $u:V\to V$ unitary and $R\in\CR(V)$, one has
\begin{align}\label{eq:exampleequivalences}
    \begin{aligned}
        &(u\ot u)R(u^*\ot u^*) \in \CR(V),\qquad &(u\ot u)R(u^*\ot u^*)&\sim R,
        \\
        [u\ot u,R]=0
        \Rightarrow
        &(1\ot u)R(1\ot u^*) \in \CR(V),\qquad &(1\ot u)R(1\ot u^*)&\sim R.
    \end{aligned}
\end{align}
To prove these claims, one checks that in the first case, the tensor power $u^{\ot n}$ is a unitary intertwiner for $\rho_R^{(n)}\cong\rho_{(u\ot u)R(u^*\ot u^*)}^{(n)}$, and in the second case, $u\ot u^2\ot\ldots\ot u^n$ is a unitary intertwiner for $\rho_R^{(n)}\cong\rho_{(1\ot u)R(1\ot u^*)}^{(n)}$ \cite{ContiLechner:2019}.

\bigskip

The challenging general problem of classifying $\CR/\!\!\sim$ is currently open. It is however possible to restrict to subfamilies of $\CR$ that are compatible with $\sim$, for instance by restricting the dimension or the spectrum of $R$. The family of all unitary R-matrics of dimension $\dim R=2$ can be completely classified up to equivalence, as will be demonstrated at the end of this section. Already for $\dim V=3$, no comparable classification result is known (see, however, \cite{HietarintaMartinRowell:2024}, for a classification of the charge-conserving R-matrices in dimension three). Part of the problem is, as seen in Proposition~\ref{proposition:2d}, R-matrices with various different spectra and different cardinalities of spectra appear, which leads to many equivalence classes.

This observation suggests that instead of prescribing the dimension, one might consider R-matrices with prescribed spectrum. Apart from the trivial case $\sigma(R)=\{q\}$ of R-matrices $R=q\cdot1$ that are multiples of the identity, the classification is understood in the family of all unitary R-matrices with spectrum $\sigma(R)\subset\{1,-1\}$, i.e. the unitary \emph{involutive} R-matrices (of arbitrary dimension). In this case, the representations $\rho_R^{(n)}$ factor through the symmetric group $S_n$. The involutive unitary R-matrices have been classified in \cite{LechnerPennigWood:2019} by using extremal characters of the infinite symmetric group. One finds an explicit bijection with the set of pairs of Young diagrams, and for each equivalence class, a representative can be specified.

\bigskip

In this paper, we go one step further and describe the more general classification problem for unitary R-matrices $R$ with two eigenvalues, i.e. $|\sigma(R)|=2$. This family includes all unitary Yang-Baxter representations of the Temperley-Lieb and Hecke algebras, underlying the Jones polynomial. As we will see, it contains non-involutive elements which can be seen as ``non-classical deformations'' of symmetries (involutive R-matrices).

Without the restriction to unitary R-matrices, this family has been considered by Gurevich \cite{Gurevich:1993} and Ram \cite{Ram:1991}. The restriction to unitary R-matrices has however a quite effect. As described by Rowell and Wang \cite{RowellWang:2012} (see also \cite{DelaneyRowellWang:2016}), most R-matrices are not unitary, and unitary local (in the sense of tensor products) braid group representations are rare. As unitary R-matrices with two eigenvalues can also be described in terms of Hecke or Temperley-Lieb algebras, one may equivalently study the classification question in terms of certain representations of these algebras. In a concrete matrix setting, this question has been addressed by Bytsko \cite{Bytsko:2015,Bytsko:2022}.

In contrast, we will here use the subfactor framework developed with Conti in \cite{ContiLechner:2019}. We will recall the relevant notions in Section~\ref{section:framework}, focussing in particular on Markov traces. There we also recall a family of unitary R-matrices known as \emph{Gaussian} R-matrices, going back to work of Goldschmidt and Jones \cite{GoldschmidtJones:1989,Jones:1989,Jones:1989_2} and later generalized by Galindo and Rowell \cite{GalindoRowell:2013}.

Whereas Section~\ref{section:framework} still applies to arbitrary unitary R-matrices, we specialize to R-matrices with two eigenvalues in Section~\ref{section:hecke}, fixing without loss of generality the spectrum to be
\begin{align}
    \sigma(R)=\{-1,q\}, \qquad q\in\bT\setminus\{-1\}.
\end{align}
As the involutive family, corresponding to $q=1$, has already been classified and behaves very differently from $q\neq1$, we restrict to the non-involutive case $q\neq\pm1$.

Applying the results from the general framework, we show that the equivalence classes of unitary non-involutive R-matrices with two eigenvalues are labelled by three data: The eigenvalue~$q$, the normalized trace $\eta\in(0,1)$ of the spectral projection corresponding to eigenvalue $-1$, and the dimension $d\in\Nl$. These three data $q,\eta,d$ are not independent, and the task is to understand which triples $(q,\eta,d)$ are realized by R-matrices.

In this context, a main result is Theorem~\ref{thm:class1}, which builds on Wenzl's work on Hecke algebra subfactors \cite{Wenzl:1988_2} and a rationality property of braid group characters (Proposition~\ref{proposition:tauRelementary}), and establishes strong constraints on the possible triples $(q,\eta,d)$: Only eight families of equivalence classes, each with fixed $q$ and $\eta$, but varying dimension $d$, remain. In particular, these classes are empty unless $q\in\{i,-i,e^{i\pi/3},e^{-i\pi/3}\}$. When restricting to the Temperley-Lieb subfamily, this recovers a theorem of Bytsko \cite{Bytsko:2019}.

By exploiting Temperley-Lieb structure and symmetries, these eight families can be described by just three families. We then show in Proposition~\ref{thm:unfinished} that two of these three families can be exhausted by Gaussian R-matrices and tensor products with identities. The last family of classes, corresponding to $q=e^{i\pi/3}$, $\eta=\frac12$, and arbitrary even dimension $d$, is not yet understood. We show that it is empty for $d=2$, but do not have a proof that this remains true for $d>2$. R-matrices of this type would need to represent the Hecke algebra but not the Temperley-Lieb algebra, and correspond to an unkown localization of a specific braid group representation \cite{GalindoHongRowell:2011}.

Some of our results were already presented in the Oberwolfach proceedings \cite{Lechner:2026}.

\bigskip
\bigskip

To conclude the introduction and for later reference, we describe the simple classification of unitary R-matrices of dimension two.

\begin{proposition}\label{proposition:2d}
    Every unitary R-matrix $R$ with dimension $\dim R=2$ is equivalent to precisely one of the following matrices\footnote{Here, and only in this proposition and its proof, the notation $R_1,\ldots,R_4$ refers to the specified matrices and not to \eqref{eq:Rk}.} (with the usual identification $M_2\ot M_2\cong M_4$):
    \begin{alignat}{3}
    R_1&=q\cdot1,&\qquad q&\in\bT,
    \\
    R_2
    &=
    \left(
        \begin{array}{cccc}
        p\\
        &&q\\
        &q\\
        &&&s
        \end{array}
    \right),
    &\qquad
    q&\in\bT,\quad \{p,s\}\in\bT^2,\quad p\neq s,
    \\
    R_3
    &=
    \left(
        \begin{array}{cccc}
        &&&q\\
        &p\\
        &&p\\
        q
        \end{array}
    \right),
    &\qquad p,q&\in\bT,
    \\
    R_4
    &=
    \frac{q}{\sqrt{2}}
            \left(
                \begin{array}{rrrr}
                1 & 1 \\
                -1 & 1 \\
                &&1&-1\\
                &&1 & 1
                \end{array}
            \right),
            &\qquad
            q&\in\bT,
            \label{R4}
    \end{alignat}
    with uniquely fixed parameters $q,p$ or unordered parameter pairs $\{p,s\}$.
\end{proposition}
\begin{proof}
    Building on earlier work by Hietarinta and Dye \cite{Hietarinta:1992_8,Dye:2003_2}, it was shown in \cite{ContiLechner:2019}, that every unitary R-matrix is equivalent to one of the R-matrices listed below:
    \begin{alignat}{3}
    R_1&=q\cdot1,
    & \sigma(R_1)&=\{q\},
    \\
    R_2
    &=
    \left(
        \begin{array}{cccc}
        p\\
        &&q\\
        &r\\
        &&&s
        \end{array}
    \right),
    &
    \sigma(R_2)&=\{p,s,\sqrt{qr},-\sqrt{qr}\},
    \\
    R_3
    &=
    \left(
        \begin{array}{cccc}
        &&&q\\
        &p\\
        &&p\\
        r
        \end{array}
    \right),
    &
    \sigma(R_3)&=\{p,\sqrt{qr},-\sqrt{qr}\},
    \\
    R_4
    &=
    \frac{q}{\sqrt{2}}
            \left(
                \begin{array}{rrrr}
                1 & 1 \\
                -1 & 1 \\
                &&1&-1\\
                &&1 & 1
                \end{array}
            \right),
            &\qquad
            \sigma(R_4)&=\left\{\tfrac{q(1+i)}{\sqrt{2}},\tfrac{q(1-i)}{\sqrt{2}}\right\}.
            \label{R4}
    \end{alignat}
    To proceed from this fact to the similar but stronger statement of the proposition, we need to show that we may restrict to the more specific parameter choices given in the statement of the proposition, and that no two distinct R-matrices as listed there are equivalent.

    We first consider the spectra of this R-matrices, displayed above. Clearly $R_1$ is not unitarily equivalent (and in particular, not equivalent) to any of the other R-matrices, and $R_4$ is not unitarily equivalent to $R_2$ or $R_3$ because the latter two have a pair of opposite eigenvalues $\pm\la$, which $R_4$ does not have. To compare $R_2$ and $R_3$, we use the equivalences \eqref{eq:exampleequivalences}. Consider the unitary diagonal $(2\times2)$-matrix $u=\diag((r/q)^{1/4},1)$, where $q,r$ are the parameters of $R_3$. One computes that $(u\ot u)R_3(u^*\ot u^*)$ differs from $R_3$ only by replacing $r$ and $q$ by $\sqrt{rq}$. As $(u\ot u)R_3(u^*\ot u^*)\sim R_3$, we may restrict to parameters $q=r$ in $R_3$ up to equivalence.

    Similarly, one checks that $R_2$ commutes with tensor squares $u\ot u$ of diagonal unitary $(2\times2)$-matrices. Hence $R_2\sim(1\ot u)R_2(1\ot u^*)$, and as $(1\ot u)R_2(1\ot u^*)$ has the same form as $R_2$, one verifies that $u$ can be chosen to match $q$ and $r$. Hence also in $R_2$, we may restrict to $r=q$ up to equivalence.

    For $p\neq s$ in $R_2$, the matrices $R_2$ and $R_3$ are not unitarily equivalent (2 vs. 3 eigenvalues). For $p=s$, they are unitarily equivalent (with identified parameters), and in fact even $R_2\sim R_3$ in this case. To prove that, we consider the unitary $u=\genfrac{(}{)}{0pt}{1}{0\;\;a}{1\;\;0}$ with $a=\sqrt{p/q}$. By straightforward computation, one checks that $u\ot u$ commutes with $R_3$, and that $R_3\sim (1\ot u)R_3(1\ot u^*)=R_2$.

    The equivalence class of $R_2$ depends only on $q$ and the unordered pair $\{p,s\}$ because conjugating $R_2$ with the tensor square of $\genfrac{(}{)}{0pt}{1}{0\;\;1}{1\;\;0}$ swaps $p$ and $s$.
\end{proof}

\section{R-matrices, subfactors, and Markov traces}\label{section:framework}

Before entering the discussion of R-matrices with two eigenvalues, we derive a few general results that will be helpful later. To begin with, we recall the subfactor approach to unitary R-matrices developed in \cite{ContiLechner:2019}.

Given $R\in\CR(V)$, we write $d:=\dim V$ for its dimension and consider the infinite tensor product of matrix algebras $M_d\cong B(V)$, namely $\CN=\bigotimes_{n\in\Nl}M_d$, considered as a von Neumann algebra that is weakly closed w.r.t. the topology induced by the normalized trace $\tau$. Here $\tau$ is defined on finite tensor products as
\begin{align}
    \tau(A_1\ot\ldots\ot A_n)
    :=
    \prod_{k=1}^n\frac{\Tr(A_k)}{d},\qquad A_k\in M_d.
\end{align}
This algebra has the canonical $\tau$-preserving shift endomorphism $\varphi$ fixed by
\begin{align}
    \varphi(A_1\ot\ldots\ot A_n\ot 1\ldots)=1\ot A_1\ot\ldots A_d\ot1\ldots.
\end{align}
Its $\tau$-preserving left inverse is usually denoted $\phi:\CN\to\CN$ and satisfies
\begin{align}\label{eq:phiBimodule}
    \phi(\varphi(x)y\varphi(z))
    =
    x\phi(y)z,\qquad x,y,z\in\CN.
\end{align}
It acts as a normalized partial trace in the first tensor factor, namely $\phi(A_1\ot A_2\ot\ldots)=\tau(A_1)\cdot A_2\ot\ldots$.

With the help of $\phi$, we can describe an important invariant for the equivalence relation $\sim$, namely the (normalized) partial trace of $R$. For $R,S\in\CR$, we have \cite[Thm.~5.8]{ContiLechner:2019}
\begin{align}
    R\sim S\;\Longrightarrow\; \phi(R)=\phi(S).
\end{align}
The reverse implication does not hold in general (also not when besides $\phi(R)=\phi(S)$ also $\dim(R)=\dim(S)$ and $\sigma(R)=\sigma(S)$ is assumed).

Note that although $\phi$ is defined with operator-algebraic methods, $\phi(R)$ is an elementary object that can be computed with tools from linear algebra. With $R\in M_d\ot M_d$, it can be identified with the normalized partial trace $\phi(R)=\frac{1}{\dim R}(\Tr\ot\id)(R)\in M_d$.

\medskip

Having recalled these data, the representations $\rho_R^{(n)}$ of $B_n$ given by $R\in\CR(V)$ can be assembled to a unitary representation $\rho_R$ of the infinite braid group $B_\infty$ (presented on generators $b_k$, $k\in\Nl$, with the same relations \eqref{eq:BnRelations} as the $B_n$) given by
\begin{align}
    \rho_R:B_\infty\to\End\CN,\qquad b_k\mapsto \varphi^{k-1}(R),\quad k\in\Nl.
\end{align}
In the following, we will frequently identify the braid groups $B_n$ on $n$ strands with their images in the inductive limit $B_\infty$.

The counterpart of the canonical endomorphism on the level of $B_\infty$ is the injective group homomorphism
\begin{align}
    s:B_\infty\to B_\infty,\qquad b_k\mapsto b_{k+1},
\end{align}
which clearly satisfies $\rho_R\circ s=\varphi\circ\rho_R$.

It is often also useful to consider the von Neumann subalgebra $\CL_R\subset\CN$ generated by $\rho_R$, i.e. the smallest von Neumann algebra inside $\CN$ that contains all $\varphi^n(R)$, $n\in\Nl_0$. We note that $\CN$ and $\CL_R$ are type II${}_1$ factors, and
\begin{align}
    \begin{array}{ccc}
        \varphi(\CN) &\subset &\CN
        \\
        \cup & & \cup
        \\
        \varphi(\CL_R) &\subset &\CL_R
    \end{array}
\end{align}
is a commuting square of type II${}_1$ factors \cite[]{ContiLechner:2019}. The relative commutant $\varphi(\CN)'\cap\CN$ consists of all elements of the form $x\ot1\ot1\ot\ldots$, $x\in M_d$, and the relative commutant $\varphi(\CL_R)'\cap\CL_R\subset\varphi(\CN)'\cap\CN$ contains in particular the partial trace $\phi(R)$.

We define the \emph{character} of $R$ as
\begin{align}
    \tau_R := \tau\circ\rho_R : B_\infty \to \Cl,
\end{align}
which we also consider as a functional on the unital group ${}^*$-algebra $\Cl B_\infty$. It is important to note that given $R,S\in\CR$, we have the equivalence \cite[Prop.~5.2]{ContiLechner:2019}
\begin{align}\label{eq:EquivalenceAndCharacters}
    R\sim S
    \quad\Longleftrightarrow\quad
    \tau_R=\tau_S\quad\text{and}\quad\dim R=\dim S
    .
\end{align}
In view of this result, it is natural to consider the character $\tau_R$ of $R$. Some elementary properties are gathered in the following proposition.

\begin{proposition}\label{proposition:tauRelementary}
    The character $\tau_R$ of a unitary R-matrix $R\in\CR(V)$ has the following properties:
    \begin{enumerate}
        \item\label{item:tauRnormalized} $\tau_R$ is a \emph{normalized trace} on $\Cl B_\infty$, i.e. $\tau_R:\Cl B_\infty\to\Cl$ is a linear functional such that $\tau_R(xy)=\tau_R(yx)$ for all $x,y\in\Cl B_\infty$, and $\tau_R(1)=1$,
        \item\label{item:tauRpositive} $\tau_R$ is \emph{positive and faithful}, i.e. $\tau_R(x^*x)>0$ for all $x\in \Cl B_\infty\setminus\{0\}$,
        \item\label{item:tauRfactorizes} $\tau_R$ \emph{factorizes} in the sense that
        \begin{align}
            \tau_R(xs^n(y))=\tau_R(x)\tau_R(y),\qquad x\in \Cl B_n,\,y\in\Cl B_\infty.
        \end{align}
        \item\label{item:tauRrational} $\tau_R$ is \emph{rational} in the following sense: There exists $d\in\Nl$ (namely, $d=\dim R$) such that for any $n\in\Nl$ and any orthogonal projection $p\in\Cl B_n$,
        \begin{align}
            d^n\tau_R(p)\in\Nl_0.
        \end{align}
    \end{enumerate}
\end{proposition}
\begin{proof}
    \ref{item:tauRnormalized} is a straightforward consequence of $\rho_R$ being a representation and $\tau$ a trace, and \ref{item:tauRpositive} follows similarly because $\rho_R$ is a ${}^*$-representation of $\Cl B_\infty$ and $\tau$ is positive and faithful.

    The normalized trace $\tau$ factorizes over tensor products, which implies property \ref{item:tauRfactorizes} of $\tau_R$ by noting that $s^n(y)$ acts trivially on the first $n$ tensor factors, and $x\in\Cl B_n$ acts non-trivially only on the first $n$ factors.

    \ref{item:tauRrational} follows because for $p\in \Cl B_\infty$ any orthogonal projection, $\rho_R(p)=\rho_R^{(n)}(p)\ot1\ldots\in\CN$ is an orthogonal projection as well. Hence
    \begin{align}
        (\dim R)^n\tau_R(p)=\Tr_{V\tp{n}}(\rho_R^{(n)}(p))=\dim(\rho_R^{(n)}(p)V\tp{n})\in\Nl_0.
    \end{align}
\end{proof}

In addition to the properties listed here, we will be interested in another relevant notion, namely that of a \emph{Markov trace}, satisfying a factorization property which plays a prominent role in knot theory \cite{Jones:1983_2,Jones:1987,Wenzl:1988_2,FreydYetterHosteLickorishMillettOcneanu:1985}.

\begin{definition}
    A trace $\mu$ on $\Cl B_\infty$ is called a \emph{Markov trace} if\footnote{This definition is equivalent to the condition $\mu(xb_n)=\mu(b_1)\mu(x)$ for all $x\in B_n$, $n\in\Nl$, which is more often used to define Markov traces \cite{Jones:1987,Wenzl:1988_2}. To prove the equivalence of the two conditions, one uses the fact that $B_n$ admits an inner automorphism $\gamma_n$, $n\in\Nl$, such that $\gamma_n(b_k)=b_{n-k}$ for $1\leq k\leq n-1$ \cite{Birman:1974}. Since $\mu$ is a trace, we have $\mu\circ\gamma_n=\mu$, $\mu\circ s=\mu$, $\mu(b_n)=\mu(b_1)$, and the equivalence follows.}
    \begin{align}
        \mu(b_1s(x))=\mu(b_1)\mu(x),\qquad x\in \Cl B_\infty.
    \end{align}
\end{definition}

While Proposition~\ref{proposition:tauRelementary} holds for the character $\tau_R$ of every $R\in\CR$, only some of them have the Markov property. The following result makes use of the left inverse $\phi$ of the canonical endomorphism $\varphi$ and its bimodule property \eqref{eq:phiBimodule}.

\begin{proposition}\label{prop:markov-comm}
    Let $R\in\R$. Then $\tau_R$ is a Markov trace if and only if $R$ has trivial partial trace\footnote{R-matrices have coinciding left and right partial traces \cite[Thm.~5.10]{ContiLechner:2019}.}, $\phi(R)=\tau(R)\cdot1$.
\end{proposition}
\begin{proof}
	The trace $\tau_R$ is a Markov trace if and only if for each $x\in\Cl B_\infty$, we have
    \begin{align*}
        0&=
        \tau_R(b_1s(x))-\tau_R(b_1)\tau_R(x)
        \\
        &=
        \tau(R\varphi(\rho_R(x)))-\tau(R)\tau(\rho_R(x))
        \\
        &=
        \tau\Big((R-\tau(R)\cdot1)\,\varphi(\rho_R(x))\Big)
        \\
        &=
        \tau\Big((\phi(R)-\tau(R)\cdot1)\,\rho_R(x)\Big)
        .
    \end{align*}
    By continuity, this extends to $0=\tau\Big((\phi(R)-\tau(R)\cdot1)\,m\Big)$ for all $m\in\CL_R$, and in particular we may choose $m=(\phi(R)-\tau(R)1)^*$ because $\phi(R)\in\CL_R$. As $\tau$ is faithful, we obtain $\phi(R)=\tau(R)1$. The reverse implication is obvious.
\end{proof}

Recalling that $\phi(R)$ lies in the relative commutant $\varphi(\CL_R)'\cap\CL_R$, it is clear that $\tau_R$ is Markov when the subfactor $\varphi(\CL_R)\subset\CL_R$ is irreducible. As irreducibility is usually not easy to check from a given R-matrix, we next present a straightforward sufficient condition for irreducibility. We will say that $R$ \emph{has a pair of opposite eigenvalues} if its spectrum $\sigma(R)$ contains $\la$ and $-\la$ for some $\la\in\Cl$.

\begin{proposition}\label{prop:comm-R2}
    If $R\in\R$ has no pair of opposite eigenvalues, $\varphi(\CL_R)\subset\CL_R$ is irreducible and consequently $\tau_R$ is a Markov trace.
\end{proposition}
\begin{proof}
    This proof relies on the following characterization of the relative commutant of $\varphi(\CL_R)\subset\CL_R$:
    \begin{align}
        \varphi(\CL_R)'\cap\CL_R
        =
        M_d\cap\CL_R
        \subset\{x\in M_d \col \varphi(x)=R^*xR\},
    \end{align}
    see \cite[Prop.~3.4 (i)]{ContiLechner:2019} for a proof. As $\CL_R=\CL_{R^*}$, we conclude that any $x\in\varphi(\CL_R)'\cap\CL_R$ satisfies $\varphi(x)=R^*xR=RxR^*$, i.e. $x$ commutes with $R^2$. As $R$ has no pair of opposite eigenvalues, $R^2$ and $R$ have the same spectral projections, and we find that $x$ commutes not only with $R^2$, but also with $R$. On the other hand, $x$ also commutes with $\varphi^n(R)$, $n\geq1$, because these operate trivially on the first tensor factor. Thus any $x\in\varphi(\CL_R')\cap\CL_R$ commutes with the whole representation $\rho_R$, i.e. $x$ lies in the center of $\CL_R$. As $\CL_R$ is a factor, $x$ must be a multiple of the identity.
\end{proof}

The converse of this proposition is not true, i.e. there exists $R\in\R$ such that $\varphi(\CL_R)\subset\CL_R$ is irreducible and $R$ has a pair of opposite eigenvalues. An example of this is the tensor flip \cite{LechnerPennigWood:2019}. However, for the class of non-involutive Hecke R-matrices, to be introduced next, Prop.~\ref{prop:comm-R2} will turn out to be an efficient tool.

\bigskip

As an example and for later reference, we now recall an interesting family of unitary R-matrices with Markov traces. These R-matrices are called metaplectic or Gaussian. They go back to work of Goldschmidt and Jones \cite{GoldschmidtJones:1989,Jones:1989,Jones:1989_2} and have been generalized by Galindo and Rowell \cite{GalindoRowell:2013}. We will define one R-matrix $G_d$ per dimension $d\geq2$.


For the definition of the Gaussian R-matrices, we will have to distinguish between odd and even $d$. Following \cite{GalindoRowell:2013}, we define the root of unity
\begin{align}
 \xi:=
 \begin{cases}
  e^{2\pi i/d} & d \text{ odd }\\
  e^{\pi i/d} & d \text{ even }
 \end{cases}
\end{align}
and fix an orthonormal basis $(e_k)_{k=0,\ldots,d-1}$ of $\Cl^d$. We define $U\in M_d\ot M_d$ in the tensor product basis as
\begin{align}
 U(e_k\ot e_l)
 :=
 \xi^{l-k}(e_{k+1}\ot e_{l+1}),
\end{align}
where here and in the following, the indices on $e_k$ are extended periodically, i.e. $e_{k+d}=e_k$. With these definitions, one verifies that for any (even or odd) $d\in\Nl$, the crucial exchange relation
\begin{align}
 U\varphi(U)=\xi^2\varphi(U)U,\qquad U^d=1
\end{align}
holds. This implies that
\begin{align}
 G_d:=\frac{1}{\sqrt{d}}\sum_{k=0}^{d-1}\xi^{k^2}U^k
\end{align}
is a unitary R-matrix, which we will refer to as the {\em Gaussian R-matrix of dimension $d$}.

\begin{lemma}\label{lemma:Gauss}
	Let $G_d\in\CR(\Cl^d)$ be the Gaussian R-matrix of dimension $d$.
	\begin{enumerate}
	 \item\label{item:GaussianRSpectrum} $G_d$ has at most $d$ distinct eigenvalues, and all eigenvalues are $(4d)$-th roots of unity. More precisely, the eigenvalues of $G_d$ are
    \begin{align*}
        \mu_l(d) &=
                \begin{cases}
                    e^{i\pi/4}e^{-i\pi l^2/d} & d \text{ even }\\
                    \frac{1}{\sqrt{2}}e^{i\pi/4-\frac{i\pi l^2}{2d}}(1+e^{-i\pi(d+2l)/2}) & d \text{ odd}
                \end{cases}
                ,\qquad l=0,\ldots,d-1.
    \end{align*}
	Specifically for $d=2$ and $d=3$, the spectrum is
	\begin{align}
        \sigma(G_2)&=\{e^{i\pi/4},e^{-i\pi/4}\},\\
        \sigma(G_3)&=\{i,e^{-i\pi/6}\}.
	\end{align}

	\item\label{item:GaussianRPartialTrace} The normalized partial trace of $G_d$ is
	\begin{align}
		\phi(G_d)=\frac{1}{\sqrt{d}},
	\end{align}
	so $\tau_{G_d}$ is a Markov trace.
	\end{enumerate}
\end{lemma}
\begin{proof}
	\ref{item:GaussianRSpectrum} Since $U^d=1$, its eigenvalues are $d$-th roots of unity. To see that all $d$-th roots are eigenvalues of $U$, one checks directly that the $P_l:=\frac{1}{d}\sum_{k=0}^{d-1}e^{-\frac{2\pi i}{d}kl}U^k$, $l=0,\ldots,d-1$, are non-zero orthogonal projections with $P_lP_m=\delta_{lm}P_l$, $\sum_l P_l=1$, and $\sum_l e^{\frac{2\pi i l}{d}}P_l=U$. This identifies the $P_l$ as the spectral projections of $U$ and shows $\sigma(U)=\{e^{\frac{2\pi il}{d}}\,:l=0,\ldots,d-1\}$.

    The eigenvalues of $G_d$ are therefore the numbers
    \begin{align}
        \mu_l
        =
        \frac{1}{\sqrt{d}}\sum_{k=0}^{d-1}\xi^{k^2}e^{\frac{2\pi ilk}{d}},
    \end{align}
    in particular $|\sigma(G_d)|\leq d$.

    By application of the reciprocity theorem for generalized Gauss sums \cite[Thm.~1.2.2]{BerndtEvansWilliams:1998} and recalling that $\xi=e^{2\pi i/d}$ for $d$ odd and $\xi=e^{i\pi/d}$ for $d$ even, it follows that
    \begin{align}\label{eq:GaussianEigenvalues}
        \mu_l &=
                \begin{cases}
                    e^{i\pi/4}e^{-i\pi l^2/d} & d \text{ even }\\
                    \frac{1}{\sqrt{2}}e^{i\pi/4-\frac{i\pi l^2}{2d}}(1+e^{-i\pi(d+2l)/2}) & d \text{ odd}
                \end{cases}
                .
    \end{align}
    An explicit check shows that $\mu_l^{4d}=1$ in all cases.

	\ref{item:GaussianRPartialTrace} It follows directly from the definition of $U$ that the partial traces of $U,U^2,\ldots,U^{d-1}$ vanish. Hence $\phi(G_d)=\frac{1}{\sqrt{d}}\sum_{k=0}^{d-1}\xi^{k^2}\phi(U^k)=\frac{1}{\sqrt{d}}$.
\end{proof}

\section{Unitary Hecke R-matrices}\label{section:hecke}

As motivated in the introduction, we now restrict attention to unitary R-matrices with prescribed spectrum. That is, we fix a finite subset $S\subset\bT$ of the unit circle and demand $\sigma(R)=S$. Then $\prod_{\la\in S}(R-\la)=0$, i.e. the homomorphism $\rho_R:\Cl B_\infty\to\CN$ factors through the quotient $(\Cl B_\infty)_S=\Cl B_\infty/\CI(S)$ by the ${}^*$-ideal $\CI(S)$ generated by the polynomial $\prod_{\la\in S}(b_1-\la)$.

If $S=\{\la\}$, the only unitary R-matrix with spectrum $S$ is trivially $R=\la1$. In the case of two distinct eigenvalues, $S=\{\la_1,\la_2\}$ with $\la_1\neq\la_2$, the situation becomes much more interesting. We first note that since the Yang-Baxter equation is homogeneous, multiples of R-matrices are again R-matrices, and two R-matrices $R,S$ with spectrum $\{\la_1,\la_2\}$ are equivalent if and only if $-\la_2^{-1}R\sim-\la_2^{-1}S$ are equivalent. That is, we may restrict to R-matrices having spectrum $\{-1,q\}$, where $q:=-\la_1/\la_2\in\Tl\backslash\{-1\}$. In the following, we write $\CR_q(V)$ and $\CR_q:=\bigcup_V\CR_q(V)$ for the sets of all unitary R-matrices with spectrum $\sigma(R)=\{-1,q\}$ and base space $V$ or arbitrary (finite-dimensional) base space, respectively.

The elements of $\CR_q$ will be referred to as \emph{Hecke R-matrices} (with parameter $q$) because the algebra $H_\infty(q):=(\Cl B_\infty)_{\{-1,q\}}$ is known as the (infinite) Hecke algebra (of type $A$).

A special case is $q=1$. This case corresponds to $R$ having the opposite pair of eigenvalues $\{-1,+1\}$ and thus being involutive, and has been classified before \cite{LechnerPennigWood:2019}. In the non-involutive case, the results of the previous section immediately imply the following lemma.

\begin{lemma}\label{lemma:RMarkov}
    Let $R\in\R_q$ with $q\neq\pm1$. Then $\varphi(\CL_R)\subset\CL_R$ is irreducible, and $\tau_R$ is a positive Markov trace on $H_\infty(q)$.
\end{lemma}
\begin{proof}
    $R$ does not have an opposite pair of eigenvalues because $q\neq1$, which implies irreducibility and in particular the Markov property of $\tau_R$ by Prop.~\ref{prop:comm-R2}. The trace $\tau_R$ is positive by construction.
\end{proof}

Given $q\neq-1$, let us recall the structure of the Hecke algebra $H_\infty(q)$. The minimal polynomial for spectrum $\{-1,q\}$ is $0=(b_1+1)(b_1-q)$, which implies that in $H_\infty(q)=(\Cl B_\infty)_{\{-1,q\}}$ we have, $k\in\Nl$,
\begin{align}
    b_k=q-(1+q)e_k,\qquad 
    e_k:=\frac{q-b_k}{q+1},\qquad q\neq-1.
\end{align}
It is a simple calculation to check that $H_\infty(q)$ is the unital ${}^*$-algebra with generators $1,e_k$, $k\in\Nl$, and relations
\begin{align}\label{eq:Hecke-e}
	e_ke_{k+1}e_k-\frac{q}{(1+q)^2} e_k
	&=
	e_{k+1}e_ke_{k+1}-\frac{q}{(1+q)^2} e_{k+1}
	,\\
	e_ke_l
	&=
	e_le_k\quad\text{for } \;|k-l|>1,\\
	e_k^2&=e_k=e_k^*.
\end{align}
For an R-matrix $R\in\CR_q$ with spectral decomposition $R=-P+q(1-P)$ (i.e., $P$ is the spectral projection corresponding to eigenvalue $-1$, and $1-P$ is the spectral projection corresponding to eigenvalue $q$), the representation $\rho_R$ of $H_\infty(q)$ is fixed by
\begin{align}
    \rho_R(e_1)=P.
\end{align}
We will also use the generators $e_k^\perp:=1-e_k$ which satisfy the same relations as the $e_k$.

Traces and representations of Hecke algebras have been thoroughly investigated because of their relation to the Jones polynomial and its generalizations, and to subfactors \cite{FreydYetterHosteLickorishMillettOcneanu:1985,Jones:1987,Wenzl:1988_2}. It is known that for every $\eta\in\Cl$, there exists a unique Markov trace $\mu_\eta$ on $H_\infty(q)$ with $\mu_\eta(e_1)=\eta$ \cite{FreydYetterHosteLickorishMillettOcneanu:1985,Jones:1987}. In particular, a Markov trace $\mu$ on $H_\infty(q)$ is uniquely fixed by its value $\eta=\mu(e_1)$ on $e_1$.

\begin{corollary}
    Let $q\in\bT\backslash\{1,-1\}$ and $R,S\in\R_q$. The following are equivalent:
    \begin{enumerate}
        \item\label{item:RS} $R\sim S$.
        \item\label{item:tauMarkov} $\tau_R(e_1)=\tau_S(e_1)$ and $\dim R=\dim S$.
        \item\label{item:trace} $\Tr R=\Tr S$ and $\dim R=\dim S$.
    \end{enumerate}
\end{corollary}
\begin{proof}
    The implication \ref{item:RS}$\Rightarrow$\ref{item:tauMarkov} is clear, and the equivalence of \ref{item:tauMarkov} and \ref{item:trace} follows from the spectral decomposition $R=-P+q(1-P)$, which yields
    \begin{align}
        \tau_R(e_1)=\tau(P)=\frac{1}{1+q}\left(q-\frac{\Tr(R)}{(\dim R)^2}\right).
    \end{align}
    The implication \ref{item:tauMarkov}$\Rightarrow$\ref{item:RS} follows from the uniqueness result on Markov traces mentioned above, and \eqref{eq:EquivalenceAndCharacters}.
\end{proof}

We may therefore label the equivalence classes of Hecke type R-matrices by triples of  the form
\begin{align}\label{eq:label}
    (q,\eta,d) \in \bT\setminus\{\pm1\}\times(0,1)\times\Nl,
\end{align}
with the entries corresponding to the second eigenvalue $q$ (besides $-1$) of $R$, the value $\tau_R(e_1)=\tau(P)$ of the normalized trace on the spectral projection of eigenvalue $-1$, and the dimension $d=\dim R$. We will write $[q,\eta,d]\in\CR_q/\!\!\sim$ for the equivalence class of R-matrices with these parameters. Note that $[q,\eta,d]$ may be empty: The task is to find out which triples $(q,\eta,d)$ are realized by R-matrices, i.e. to decide when $[q,\eta,d]$ is not empty. Some simple restrictions are $d\geq2$ because $R$ has two distinct eigenvalues, and $\eta\in\mathbb{Q}$ because of the rationality property in Proposition~\ref{proposition:tauRelementary}~\ref{item:tauRrational}, but we will find much stronger constraints below.

\medskip

A crucial aspect in this analysis are positivity properties. For a Hecke R-matrix $R\in\R_q$, the $B_\infty$-representation $\rho_R$ defines a $C^*$-representation of $H_\infty(q)$ (denoted by the same symbol) by sending $e_k$ to the spectral projection of $\varphi^{k-1}(R)$ with eigenvalue $-1$. As shown above, these representations admit a positive Markov trace and are non-trivial in the sense that $\rho_R(e_1)\neq\rho_R(e_2)$. Independently of R-matrices, Wenzl analyzed when $H_\infty(q)$ admits non-trivial $C^*$-representations, and for which values of $\mu(e_1)$, the Markov trace $\mu$ on $H_\infty(q)$ is positive.

\begin{theorem}{\cite{Wenzl:1988_2}}\label{thm:Wenzl}
	Let $q\in\Tl\setminus\{1,-1\}$.
     \begin{enumerate}
         \item\label{item:WenzlSpectrum} $H_\infty(q)$ admits non-trivial $C^*$-representations if and only if $q=e^{\pm2\pi i/\ell}$ for some $\ell\in\Nl$ with $\ell\geq4$.
         \item\label{item:WenzlTraces} Let $q=e^{\pm2\pi i/\ell}$, $\ell\geq4$. Then there exist $\ell-1$ distinct positive Markov traces on $H_\infty(q)$. They are parameterized by $k\in\{1,\ldots,\ell-1\}$ such that $\mu(e_1)=\eta_{\ell,k}$, where
	 \begin{align}\label{eq:eta-kl}
		\eta_{\ell,k}:=\frac{1-q^{-k+1}}{(1+q)(1-q^{-k})}=\frac{\sin\frac{\pi(k-1)}{\ell}}{2\cos\frac{\pi}{\ell}\sin\frac{\pi k}{\ell}}.
	 \end{align}
    \end{enumerate}
\end{theorem}

In combination with our previous findings, this result can be used to derive the following classification theorem for non-trivial unitary Hecke R-matrices.

\begin{theorem}\label{thm:class1}
    The only possibly non-empty equivalence classes of Hecke R-matrices are the eight families of the form
    \begin{align}
        \left[\pm i,\tfrac12,2d\right],
        \quad
        \left[e^{\pm i\pi/3},\tfrac13,3d\right],
        \quad
        \left[e^{\pm i\pi/3},\tfrac23,3d\right],
        \quad
        \left[e^{\pm i\pi/3},\tfrac12,2d\right],
    \end{align}
    with $d\in\Nl$.
\end{theorem}
\begin{proof}
    We first show that $\CR_q$ is empty unless $q\in\{\pm i,e^{\pm i\pi/3}\}$. Let $R\in\R_q$. Since $R$ defines a non-trivial $C^*$-representation of $H_\infty(q)$, we have $q=e^{\pm 2\pi i/\ell}$ for some $\ell\in\Nl,\ell\geq4$ by Thm.~\ref{thm:Wenzl}~\ref{item:WenzlSpectrum}. By Lemma~\ref{lemma:RMarkov} and Thm.~\ref{thm:Wenzl}~\ref{item:WenzlTraces} we conclude that there exists $k\in\{1,\ldots,\ell-1\}$ such that $\tau_R(e_1)=\eta_{\ell,k}$ \eqref{eq:eta-kl}.

    The idea of the following argument is to exploit the rationality property (Proposition~\ref{proposition:tauRelementary}~\ref{item:tauRrational}) of $\tau_R$: If $\rho_R(p)$ is an orthogonal projection for some $p\in H_\infty(q)$, then there exists $n\in\Nl$ such that $\rho_R(p)\in\End V\tp{n}$ and therefore $\tau_R(p)=\tau(\rho_R(p))=(\dim R)^{-n}\Tr_{V\tp{n}}(\rho_R(p))$ is {\em rational}. In particular, $\eta_{\ell,k}=\tau_R(e_1)$, the normalized trace of a spectral projection of $R$, must be rational.
    
    We first note that the two extreme values $k=1$ and $k=\ell-1$ yield $\tau_R(e_1)=\eta_{\ell,1}=0$ and $\tau_R(e_1)=\eta_{\ell,\ell-1}=1$, corresponding to $R=q\id_{V\ot V}$ and $R=-\id_{V\ot V}$, respectively. Since we require $\sigma(R)=\{-1,q\}$, these values of $k$ are ruled out.

    To best analyse the remaining cases, we have to exhibit suitable projections in $\rho_R(H_\infty(q))$. Similar to the analysis of Fredenhagen, Rehren, Schroer \cite[Prop.~3.1]{FredenhagenRehrenSchroer:1989}, we consider the orthogonal projections $P^R_1:=1$,
    \begin{align}
        P^R_{n+1}:=\rho_R(e_1)\wedge\rho_R(e_2)\wedge\ldots\wedge\rho_R(e_n),\qquad n\in\Nl.
    \end{align}
    In \cite{FredenhagenRehrenSchroer:1989}, a recursion relation for these projections is proven which adapted to our Yang-Baxter setting reads, $1\leq n\leq \ell-2$,
    \begin{align*}
    P_{n+1}^R
    &=
    \varphi(P_n^R)-\alpha_n\cdot \varphi(P_n^R)\rho_R(e_1^\perp)\varphi(P_n^R)
    ,\qquad \alpha_n:=\frac{2\cos\frac{\pi}{\ell}\,\sin\frac{n\pi}{\ell}}{\sin\frac{(n+1)\pi}{\ell}}
    ,\\
    P_\ell^R&=
    \varphi(P_{\ell-1}^R).
    \end{align*}
    Since $\tau_R$ is a Markov trace, we have $\phi(R)=\tau(R)1$ by Prop.~\ref{prop:markov-comm} and therefore $\phi(\rho_R(e_1^\perp))=\phi(1-\frac{q-R}{q+1})=1-\eta_{\ell,k}=\tau_R(e_1^\perp)$. Applying $\phi$ to the recursion relation then gives
    \begin{align*}
        \phi(P^R_{n+1})
        &=
        (1-\alpha_n\,\tau_R(e_1^\perp))P_n^R,\qquad 1\leq n\leq\ell-2.
    \end{align*}
    As $\phi$ leaves $\tau$ invariant, we find
    \begin{align*}
        \tau(P^R_{n+1})
        =
        (1-\alpha_n\tau_R(e_1^\perp))\tau(P^R_n).
    \end{align*}
    Since $\eta_{\ell,k}=\tau_R(e_1)$ and $\tau(P_n^R)$, $\tau(P_{n+1}^R)$ are rational, also the coefficients $\alpha_n$ are rational. But 
    \begin{align*}
        \frac{1}{\alpha_2}=1-\frac{1}{4 \cos^2\frac{\pi}{\ell}},
    \end{align*}
    and we see that the Yang-Baxter property requires $\cos^2\frac{\pi}{\ell}\in\Ql$. For $\ell\in\Nl$, $\ell\geq4$, this is the case exactly for $\ell=4$ (with $\cos^2\frac{\pi}{4}=\frac{1}{2}$ and $q=e^{\pm 2\pi i/\ell}=\pm i$) and $\ell=6$ (with $\cos^2\frac{\pi}{6}=\frac{3}{4}$ and $q=e^{\pm i\pi/3}$).
    
    For $\ell=4$, the only value of $k$ in $\{2,\ldots,\ell-2\}$ is $k=2$, with $\eta_{4,2}=\frac{1}{2}$. For $\ell=6$, we have the three cases $k\in\{2,3,4\}$ that lead to $\eta_{6,2}=\frac{1}{3}$, $\eta_{6,3}=\frac{1}{2}$ and $\eta_{6,4}=\frac{2}{3}$.

    This proves the claimed restrictions on the possible combination of values for $q$ and $\eta$, and it remains to show the restrictions on the dimensions $d$. Let $r$ denotes the mutliplicity of the eigenvalue $-1$ of $R$, and $d=\dim R$. Then $\tau_R(e_1)=\frac{r}{d^2}$, and we see that $\tau_R(e_1)=\frac{1}{2}$ can only occur if $\dim R$ is even, whereas $\tau_R(e_1)\in\{\frac{1}{3},\frac{2}{3}\}$ can only occur if $\dim R$ is a multiple of 3.
\end{proof}

To compare this result with other results in the literature, we recall that the \emph{Temperley-Lieb algebra}~$\TL_\infty(q)$ is the quotient of $H_\infty(q)$ by the ${}^*$-ideal generated by the relations $e_ke_{k+1}e_k=\frac{q}{(1+q)^2}e_k$, $k\in\Nl$, known as the . Thus in $\TL_\infty(q)$, both sides of the Hecke relation \eqref{eq:Hecke-e} vanish. We will say that $R\in\CR_q$ is of \emph{Temperley-Lieb type} if its representation $\rho_R$ factors through $\TL_\infty(q)$, i.e. if its spectral projection of eigenvalue $-1$ satisfies
\begin{align}\label{eq:TL}
    P_1P_2P_1=
    \frac{q}{(1+q)^2}P_1.
\end{align}
As the Hecke and Temperley-Lieb relations only involve the projection $P$, we will also speak of a \emph{Hecke projection} and \emph{Temperley-Lieb} projection, respectively.

\bigskip

Parts of the classification in Theorem~\ref{thm:class1} were already known before. In the special case of dimension $\dim R=2$, Conti and Pinzari have shown that $\R_q$ is empty unless $q=\pm i$ or $\re(q)\geq\frac{1}{2}$ \cite[Prop.~9.10]{ContiPinzari:1996}. In case $q=i$, they have also shown that $R$ must necessarily be of Temperley-Lieb type. There is also a related result of Bytsko \cite[Thm.~1]{Bytsko:2019}. In our notation, he showed that under the assumption that $R$ is a non-trivial unitary R-matrix of Temperley-Lieb type, either $q=\pm i$ and $\tau_R(e_1)=\frac{1}{2}$, or $q=e^{\pm i\pi/3}$ and $\tau_R(e_1)=\frac{1}{3}$.

It is therefore of interest to know when a Hecke R-matrix is of Temperley-Lieb type, and we list a few equivalent characterizations.

\begin{lemma}\label{lemma:TL}
    Let $R\in\CR_q$, $q\neq\pm1$, and $R=-P+q(1-P)$ is spectral decomposition. Then the following are equivalent:
    \begin{enumerate}
        \item $P$ is a Temperley-Lieb projection, i.e. $P_1P_2P_1=\frac{q}{(1+q)^2}P_1$.
        \item\label{item:TLtrace} $\tau(P)=\frac{q}{(1+q)^2}$.
        \item $P_1\wedge P_2=0$.
        \item There exist $a,b\in\Cl$ such that $\tau((P_1P_2)^n)=ab^n$ for all $n\in\Nl$. (In this case, $a=b=\frac{q}{(1+q)^2}$.)
    \end{enumerate}
\end{lemma}
\begin{proof}
    We write $x:=\frac{q}{(1+q)^2}$ as a shorthand notation throughout the proof. Recall $0<x<1$.

    $a)\Leftrightarrow b)$ The Temperley-Lieb condition is equivalent to
    \begin{align*}
        0&=
        \tau((P_1P_2P_1-xP_1)^*(P_1P_2P_1-xP_1))
        \\
        &=
        \tau(P_1P_2P_1P_2P_1)-2x\tau(P_1P_2P_1)+x^2\tau(P)
        \\
        &=
        \tau(P_1P_2P_1P_2)-2x\tau(P_1P_2)+x^2\tau(P)
        \\
        &=
        \tau(P_2P_1P_2)
        -x\tau(P)+x\tau(P)^2
        -2x\tau(P)^2+x^2\tau(P)
        \\
        &=
        \tau(P)\left(1-x\right)\left(\tau(P)-x\right).
    \end{align*}
    As $P\neq0$ and $x\neq1$, this is equivalent to b).

    $a)\Rightarrow c)$ If the Temperley-Lieb condition holds, one has $\tau((P_1P_2)^{n+1})=x\tau((P_1P_2)^n)$ and hence $\tau((P_1P_2)^{n+1})=x^n\tau(P_1P_2)=x^{n+2}$. In the limit $n\to\infty$, the right hand side vanishes, and the left hand side converges to $\tau(P_1\wedge P_2)$, so that c) follows from $\tau$ being faithful.

    $c)\Rightarrow a)$ Let $v\in V\tp{3}$, and $w:=(P_1P_2P_1-xP_1)v$. Then clearly $P_1w=w$, and
    \begin{align*}
        P_2w
        &=
        (P_2P_1P_2P_1-xP_2P_1)v
        \\
        &=
        (P_1P_2P_1-xP_1+xP_2P_1-xP_2P_1)v
        =w.
    \end{align*}
    Hence $w$ lies in the range of $P_1\wedge P_2$ and is therefore $0$ by c). This implies $P_1P_2P_1=xP_1$.

    $b)\Leftrightarrow d)$ For a Hecke solution $P$ we calculate with the help of the Hecke relations and the Markov property of $\tau$
    \begin{align}\label{recurs1}
        \tau((P_1P_2)^n) &= \tau(P)\left(x\frac{1-x^{n-1}}{1-x}(\tau(P)-1)+\tau(P)\right).
    \end{align}
    For $P$ Temperley-Lieb, one has $\tau((P_1P_2)^n)=x^{n+1}$ as in the proof of $a)\Rightarrow c)$. Hence $b)$ implies $d)$. On the other hand, if $d)$ holds, i.e. $\tau((P_1P_2)^n)=ab^n$, then we see from \eqref{recurs1} $b=x$ and after a short calculation $x=\tau(P)$, i.e. $b)$.
\end{proof}

In agreement with the results of Bytsko and Conti/Pinzari, this Temperley-Lieb property is satisfied in two of the four cases listed in Prop.~\ref{thm:class1}: For $q=\pm i$ and $\tau_R(e_1)=\frac{1}{2}$, and for $q=e^{\pm i\pi/3}$ and $\tau_R(e_1)=\frac{1}{3}$, property~\ref{item:TLtrace} of Lemma~\ref{lemma:TL} is satisfied,
\begin{align}
    \frac{\pm i}{(1\pm i)^2}=\frac12,\qquad
    \frac{e^{\pm i\pi/3}}{(1+e^{\pm i\pi/3})^2}
    =
    \frac{1}{4\cos^2\frac\pi6}=\frac13.
\end{align}
For the class $\left[e^{\pm i\pi/3}, \frac{2}{3},d\right]$, we have $\frac{q}{(1+q)^2}=1-\tau_R(e_1)\neq\tau_R(e_1)$, so it is not of Temperley-Lieb type. This case can be understood as follows. On $\R_q$, we have an involutive symmetry operation that exchanges the two spectral projections of $R$. If $P$ denotes the spectral projection of $R$ corresponding to eigenvalue $-1$, then $R=-P+q(1-P)$, and
\begin{align}\label{eq:R'} 
    R':=-(1-P)+qP
\end{align}
also lies in $\R_q$: Clearly $R'$ is unitary with eigenvalues $-1$ and $q$, and the fact that it solves the Yang-Baxter equation is a consequence of the Hecke algebra relations being invariant under $e_k\leftrightarrow e_k^\perp$. On the level of equivalence classes, this implies
\begin{align}
    [q,\eta,d]'=[q,1-\eta,d],
\end{align}
so we realize that $[e^{\pm i\pi/3}, \frac{2}{3},d]$ is to the ``flipped version'' of the Temperley-Lieb class $[e^{\pm i\pi/3}, \frac{2}{3},d]$. The latter class if not of Temperley-Lieb type because the flip symmetry preserves the Hecke relations, but not necessarily the Temperley-Lieb relations.

Similarly, we note that with an R-matrix $R\in[q,\eta,d]$, its adjoint $R^*$ lies in the class
\begin{align}
    [q,\eta,d]^*=[\overline q,\eta,d].
\end{align}
This observation reduces the task of finding representatives for the classes listed in Theorem~\ref{thm:class1} to the eigenvalues $q=i$ and $q=e^{i\pi/3}$, i.e. out of the eight families listed in Theorem~\ref{thm:class1}, we may restrict without loss of generality to the three families
\begin{align}
    \left[i,\tfrac12,2d\right],\qquad
    \left[e^{i\pi/3},\tfrac13,3d\right],
    \qquad
    \left[e^{i\pi/3},\tfrac12,2d\right]
\end{align}
(with varying dimension $d\in\Nl$).

\medskip

In the following proposition, we make use of the Gaussian R-matrices $G_d$ from Lemma~\ref{lemma:Gauss} and the tensor product $R\boxtimes S\in\CR(V\ot W)$ of R-matrices $R\in\CR(V)$, $S\in\CR(W)$. It is defined as $R\otimes S$ up to reshuffling of the tensor factors from $V\ot V\ot W\ot W$ to $(V\ot W)\ot (V\ot W)$, see for example \cite[(4.15)]{LechnerPennigWood:2019}. As $R\boxtimes S$ is unitarily equivalent to $R\otimes S$, we have $\dim(R\boxtimes S)=\dim R\cdot\dim S$ and $\sigma(R\boxtimes S)=\{\la\mu\col\la\in\sigma(R),\mu\in\sigma(S)\}$.

We write $1_k:=\id_{\Cl^k\otimes\Cl^k}$ as shorthand notation.

\begin{proposition}\label{thm:unfinished}
	Let $d\in\Nl$. Then
	\begin{align}
        -e^{-i\pi/4}\,G_2\boxtimes1_d\in[i,\tfrac12,2d],\qquad\qquad
        i\,G_3\boxtimes1_d\in[e^{i\pi/3},\tfrac13,3d].
	\end{align}
\end{proposition}
\begin{proof}
    Consider the R-matrix $R:=-e^{-i\pi/4}G_2$. According to Lemma~\ref{lemma:Gauss}~\ref{item:GaussianRSpectrum}, $\sigma(R)=\{-1,i\}$, i.e. $R$ is a Hecke R-matrix with $q=i$. Denoting one spectral projection of $G_2$ by $P$, we have $G_2=e^{i\pi/4}P+e^{-i\pi/4}(1-P)$ and hence $R=i-(1+i)P$. Using Lemma~\ref{lemma:Gauss}~\ref{item:GaussianRPartialTrace}, this yields
    \begin{align*}
        \tau_R(e_1)=\tau(P)=\frac{1}{1+i}\left(i+e^{-i\pi/4}\frac{1}{\sqrt{2}}\right)=\frac{1}{2},
    \end{align*}
    and we conclude $R\in[i,\frac12,2]$. By tensoring with the identity in dimension $d$ we obtain $R\boxtimes1_d\in[i,\frac12,2d]$.

    Similarly, the R-matrix $S:=iG_3$ has spectrum $\sigma(S)=\{-1,e^{i\pi/3}\}$ according to Lemma~\ref{lemma:Gauss}~\ref{item:GaussianRSpectrum}, i.e. $S$ is a Hecke R-matrix with $q=e^{i\pi/3}$. Denoting one spectral projection of $G_3$ by $P$, we have $G_3=iP+e^{-i\pi/6}(1-P)$ and hence $S=e^{i\pi/3}-(1+e^{i\pi/3})P$. Using Lemma~\ref{lemma:Gauss}~\ref{item:GaussianRPartialTrace}, this yields
    \begin{align*}
        \tau_S(e_1)=\tau(P)=\frac{1}{1+e^{i\pi/3}}\left(e^{i\pi/3}-\frac{i}{\sqrt{3}}\right)=\frac{1}{3}.
    \end{align*}
    Again we conclude $S\in[e^{i\pi/3},\frac13,3]$ and $S\boxtimes1_d\in[e^{i\pi/3},\frac13,3d]$.
\end{proof}

The results presented so far almost provide a complete classification of all unitary R-matrices with two eigenvalues, up to the class $[e^{i\pi/3},\frac12,2d]$. It is currently not known to us whether $[e^{i\pi/3},\frac12,2d]$ is empty (for some $d$) or not. The question of existence of this R-matrix has also been raised in \cite[Sect.~5.6]{GalindoHongRowell:2011}.

It should be noted that Wenzl has constructed a $C^*$-representation of $H_\infty(e^{i\pi/3})$ with Markov trace $\mu$ satisfying $\mu(e_1)=\frac12$. This trace has all the properties listed in Proposition~\ref{proposition:tauRelementary}, including its rationality (with $d=2$). It is however not clear whether it can be represented in the tensor product form required by an R-matrix.

We also note that for $d=1$, the class $[e^{i\pi/3},\frac{1}{2},2d]$ is empty. This can be seen from the list of equivalence classes of two-dimensional R-matrices provided in Proposition~\ref{proposition:2d}: $R_1$ has only one eigenvalue, and $R_2$, $R_3$ have an opposite pair of eigenvalues. Hence $R_1,R_2,R_3$ are not of Hecke type. $R_4$ is a multiple of $G_2^*$, in particular its negative eigenvalue ratio is $-i\neq e^{\pm i\pi/3}$.

If $R\in[e^{i\pi/3},\frac{1}{2},2d]$ exists, its Yang-Baxter endomorphism \cite{ContiLechner:2019} would have index $4$ and would be standard braided \cite{ContiFidaleo:2000}.

\subsubsection*{Acknowledgements}

I would like to thank Roberto Conti, Hans Wenzl, and Eric Rowell for interesting discussions about R-matrices and subfactors. Also the financial support by the German Research Foundation DFG through the Heisenberg Grant ``Quantum Fields and Operator Algebras'' (LE 2222/3-1) is gratefully acknowledged.

\setlength{\biblabelsep}{0.4em}
\AtNextBibliography{\small}
\newrefcontext[sorting=anyt]
\setlength\bibitemsep{0.25\itemsep}
\printbibliography

\end{document}